\documentclass[12pt,reqno]{amsart}
\usepackage{amsmath, amsthm, amsopn, amssymb, microtype}
\usepackage{mathrsfs} 
\usepackage{mathscinet}
\usepackage{bbm}
\usepackage{enumerate, tikz, etoolbox, intcalc, geometry, caption, subcaption, afterpage}
\usepackage[english]{babel}
\usepackage{enumitem}
\usepackage[colorlinks=true,urlcolor=blue,pdfborder={0 0 0}]{hyperref}
\hypersetup{linkcolor=[rgb]{0,0,0.6}}
\hypersetup{citecolor=[rgb]{0,0.6,0}}
\usepackage[nameinlink]{cleveref}
\crefname{theorem}{Theorem}{Theorems}
\crefname{thm}{Theorem}{Theorems}
\crefname{mainthm}{Theorem}{Theorems}
\crefname{conj}{Conjecture}{Theorems}
\crefname{lemma}{Lemma}{Lemmas}
\crefname{lem}{Lemma}{Lemmas}
\crefname{remark}{Remark}{Remarks}
\crefname{prop}{Proposition}{Propositions}
\crefname{defn}{Definition}{Definitions}
\crefname{corollary}{Corollary}{Corollaries}
\crefname{cor}{Corollary}{Corollaries}
\crefname{section}{Section}{Sections}
\crefname{figure}{Figure}{Figures}
\crefname{quest}{Question}{Questions}

\theoremstyle{plain}
\newtheorem{theorem}{Theorem}[section]
\newtheorem{lemma}[theorem]{Lemma}

\theoremstyle{definition}
\newtheorem{definition}[theorem]{Definition}
\newtheorem{remark}[theorem]{Remark}
\newtheorem{example}[theorem]{Example}

\DeclareMathOperator{\homeo}{Homeo}
\DeclareMathOperator{\sep}{sep}
\DeclareMathOperator{\shift}{shift}
\DeclareMathOperator{\Ker}{Ker}
\DeclareMathOperator{\Fix}{Fix}
\DeclareMathOperator{\ord}{ord}
\DeclareMathOperator{\Dn}{D}

\newcommand{\Z}{\mathbb{Z}}
\newcommand{\C}{\mathcal{C}}
\newcommand{\N}{\mathbb{N}}
\newcommand{\R}{\mathbb{R}}

\newcommand{\cF}{\mathcal{F}}

\newcommand{\cC}{\mathcal{C}}
\newcommand{\cG}{\mathcal{G}}
\newcommand{\cP}{\mathcal{P}}

\title{Equivariant embedding of finite-dimensional dynamical systems}
\author{Yonatan Gutman}
\address{
    Institute of Mathematics, 
    Polish Academy of Sciences, ul. Śniadeckich 8, 00-656 Warszawa,
    Poland.
	{\tt gutman@impan.pl}
}

\author{Michael Levin} 
\address{Ben-Gurion University of the Negev,
	Department of Mathematics,
	Beer-Sheva, 8410501, Israel.
	{\tt mlevine@bgu.ac.il}
}

\author{Tom Meyerovitch}
\address{Ben-Gurion University of the Negev.
	Department of Mathematics.
	Beer-Sheva, 8410501, Israel. 
	{\tt mtom@bgu.ac.il}
}
\begin{document}
\begin{abstract}  
We prove an equivariant version of the classical Menger–Nöbeling theorem regarding topological embeddings: 
Whenever a group $G$ acts on a finite-dimensional compact metric space $X$, a generic continuous equivariant function from $X$ into $([0,1]^r)^G$ is a topological embedding, provided that for every positive integer $N$ the space of points in $X$  with orbit size at most $N$ has topological dimension strictly less than $\frac{rN}{2}$.
We emphasize that the result imposes no restrictions whatsoever on the acting group $G$ (beyond the existence of an action on a finite-dimensional space). Moreover if $G$ is finitely generated then there exists a finite subset $F\subset G$ so that  for a generic continuous map  $h:X\to [0,1]^{r}$, the map $
h^{F}:X\to ([0,1]^{r})^{F}$ given by $x\mapsto (f(gx))_{g\in F}$ is an embedding. This constitutes a 
generalization of the Takens delay embedding theorem into the topological category.
\end{abstract}
\date{\today} 
\maketitle

\section{Introduction}

Various mathematical problems in topology involve instances of the following fundamental question : Given a topological space $X$ and another topological space $Y$, when does $X$ (topologically) embed in $Y$?  According to the classical Menger–Nöbeling theorem, a compact metric space $X$ of 
(Lebesgue covering) 
dimension less than
$\frac{r}{2}$ admits a topological embedding into  $[0,1]^r$ (see \cite[Theorem 5.2]{HW41}).

In topological dynamics, the analogous fundamental embedding questions take the following form:   Given a group $G$ that acts on two topological spaces $X$ and  $Y$, when does there exist an \emph{equivariant} embedding of $X$ into $Y$, namely a continuous function $f: X \to Y$ that is a homeomorphism from $X$ onto the image $f(X) \subseteq Y$ and so that $f(g(x)) = g(f(x))$ for every $g \in G$ and every $f \in X$. In this paper, the topological spaces involved are always assumed to be compact and metrizable. So a continuous function $f: X \to Y$ defines a homeomorphism from $X$ onto the image $f(X) \subseteq Y$ if and only if it is injective.
In this paper, by a \emph{topological dynamical system} we mean a pair $(G,X)$, where $G$ is a topological group that acts by homeomorphisms on a compact metrizable space $X$.
 When $Y= ([0,1]^r)^G$ for some group $G$  is a Tychonoff cube and the action of the group $G$ on  $Y= ([0,1]^r)^G$ is the $G$-shift (see \Cref{subsec:cube_shift}) for a precise definition), the existence of an equivariant embedding of $X$ into $Y$ is equivalent to the existence of  a continuous injective mapping $f:X\to([0,1]^{d})^G$ for which it holds  $f(gx)_h=f(x)_{hg}$ for all $x \in X$ and $g,h \in G$.
The problem of equivariantly embedding into such a space $Y$, known as a ($G-$) \emph{$r$-cubical shift} has quite a long history and there is a fair amount of literature  on this problem, mostly for the case where the group $G$ is generated by a single homeomorphism (that is, $G= \Z$ or $G= \Z / m\Z$ for some $m \in \mathbb{N}$) or by finitely many commuting homeomorphisms (so that $G$ is a finitely generated abelian group). We review some of this history in the next paragraph.

In \cite{J74} Jaworski showed that for any   action of the group $G= \Z$ on any finite-dimensional compact metric space $X$ there exists an equivariant embedding of $X$ into $[0,1]^G$, under the assumption that the generator of $\Z$ corrsponds to  an \emph{aperiodic homeomoprhism} of $X$, where  a homeomorphism $T:X \to X$ is called aperiodic if $T^{n}x\neq x$ for all $x\in X$ and nonzero integer $n$. Later Nerurkar \cite{nerurkar1991} showed that the aperiodicity assumption in Jaworski's result can be weakened to the assumption that there are at most finitely  many periodic points with the same period.  
Gutman \cite{Gut16} showed that the aperiodicity assumption in Jaworski's result can be further weakened to the assumption that the set of periodic point in $X$ of period $N$ has dimension strictly less that $\frac{Nr}{2}$ for all $N$. An extension of this result  for actions of finitely generated  abelian groups was achieved by Gutman, Qiao and Szab\'{o} in \cite{GQS18}.

Our main result is a generalization of the above results where the acting group is arbitrary. Moreover as elucidated by Theorem  \ref{thm:sharpness}, our result is sharp in a strong sense. 

\begin{theorem}\label{thm:infinite_groups_Takens}
Let $(G,X)$ be a topological dynamical system where $X$ is  a finite-dimensional compact metric space.
Let $r\in \N$. Suppose that for every $N \in \N$ it holds
\begin{equation}\label{eq:periodic}
\dim(G,X)_N < \frac{rN}{2},    
\end{equation}
 where 
$$ (G,X)_N := \left\{ x \in X:~ |G \cdot x| \le N\right\} \mbox{ and } G\cdot x= \{gx : g \in G \}.
$$
Then a generic  continuous function $f:X \to [0,1]^r$ induces a $G$-equivariant topological embedding $f^{G}:X \to ([0,1]^r)^G$.
\end{theorem}

For the definition of \textit{generic}, see Subsection \ref{subsec:Genericity}.

\Cref{thm:infinite_groups_Takens}  implies in particular that if a finite group $G$ acts freely on $X$ and $\dim X < \frac{r}{2}|G|$ then $X$ embeds equivariantly into $(([0,1])^r)^G$. The case where $G$ is the trivial group coincides with the Menger–Nöbeling theorem.

\begin{remark}\label{rem:automatic_condition}
    Let $G$ be a group which does not have finite index subgroups (e.g. an infinite simple group), then Condition \eqref{eq:periodic} holds for any t.d.s $(G,X)$, $r,N\in \N$. Indeed for $x\in X$, let $\Fix_G(x)=\{g\in G|\,gx=x\}$. It is easy to see that there is a 1-1 correspondence between (left) cosets of $\Fix_G(x)$ and $Gx$.
Thus if  $|G \cdot x|<\infty$, then $\Fix_G(x)$ is of finite index in $G$. We conclude that for a group $G$ 
 which does not have finite index it holds $ (G,X)_N=\emptyset$ for all $N\in \N$. 
\end{remark}

\begin{remark}
    When $G$ is an infinite sofic group (for instance $G=\Z$),  it is not possible to remove the assumption that $X$ is finite dimensional in \Cref{thm:infinite_groups_Takens}. 
    Although in this case any compact metrizable space $X$ embeds (topologically) in $[0,1])^G$, 
    there is a further  obstruction to the existence of an equivariant embedding, namely the \emph{mean dimension} of  $(G,X)$.     
    Mean dimension is an isomorphism invariant of  topological  dynamical systems introduced
by Gromov (\cite{G}). Heuristically, whereas topological entropy measures the number
of bits per unit of time required to describe a point in a system, mean dimension
measures the required number of parameters per unit of time. The  initial systematic development of mean dimension theory  was carried out by Lindenstrauss
and Weiss in the seminar paper \cite{LW}. 
We refrain from  defining mean dimension here and refer the interested read to \cite{LW} and \cite{l12}.
Lindenstrauss and Tsukamoto 
formulated a conjecture in \cite{LT12}  regarding sufficient conditions for the existence of an equivariant embedding  of a $\Z$-dynamical system in the $\Z$-shift on $([0,1]^r)^\Z$. The  conjectured  sufficient conditions involve  mean dimension and dimensions of periodic points. For finite dimensional $\Z$-systems, the Lindenstrauss and Tsukamoto conjecture reduces to Gutman's result \cite{Gut16}. Additional cases of the conjecture were established in \cite{Gut15Jaworski,gutman2020embedding} but in full generality the conjecture is still open. 
In \cite{gutman2019application} a general embedding conjecture that generalizes the Lindenstrauss and Tsukamoto conjecture to $\mathbb{Z}^k$-actions $(X,\mathbb{Z}^k)$ ($k\in \N$) first appeared  explicitly. In the finite dimensional case, the conjecture is known to hold (\cite{GQS18}). Some additional cases of the conjecture were established in \cite{GQS18, gutman2019application} but in full generality the conjecture is still open. 
In contrast to \Cref{thm:infinite_groups_Takens}, the existing embedding theorems for infinite dimensional systems seem to rely heavily on the group structure of $\Z$ or $\Z^d$. 
Even the formulation of the embedding conjecture in \cite{gutman2019application} does not trivially extend to actions of non-commutative (say amenable) groups. For a free action $(G,X)$, it is still unknown having infinite mean dimension (when it is well-defined)  is the only additional obstruction for the existence of an equivariant embedding into $([0,1]^r)^G$ for some $r \in \mathbb{N}$.
\end{remark}

 

Complementing Remark \ref{rem:automatic_condition}, we have the following:

\begin{theorem}(Sharpness of Theorem \ref{thm:infinite_groups_Takens})\label{thm:sharpness}
    Let $N,r\in \N$. Let $G$ be a group which has a subgroup $G'$ of index $N$, then there exists a faithful t.d.s. $(G,X)$ so that  $\dim(G,X)_N = \lceil \frac{rN}{2}\rceil$ (thus Condition \ref{eq:periodic} does not hold ) and  for all  continuous functions $f:X \to [0,1]^r$ the induced map $f^{G}:X \to ([0,1]^r)^G$ is not injective.
\end{theorem}

We deduce \Cref{thm:infinite_groups_Takens} from the following theorem, our main auxiliary theorem:

\begin{theorem}\label{thm:gen_Takens}
Let $X,Y$  be compact metrizable spaces and let $\cF= \big(g_1,\ldots,g_N\big )$   be an $N$-tuple of continuous injective functions from $X$ to $Y$.
For every $f:Y \to [0,1]^r$, let $f^\cF:X \to ([0,1]^r)^N$ be given by 
\[f^\cF(x)_{i}= f(g_i(x)), ~i \in [N].\]
 For every partition $\cP$ of $\cF$ let
\begin{equation*}
X_{\cP} := \{ x\in X:~ \forall g_{i_1},g_{i_2} \in \cF,~ g_{i_1}(x)=g_{i_2}(x) ~\Leftrightarrow ~ \cP(g_{i_1})=\cP(g_{i_1})\}.
\end{equation*}
Suppose that for every partition $\cP$ of $\cF$ it holds that
\begin{equation}\label{eq:dim_inequality} \dim X_{\cP} < \frac{r}{2}|\cP|.
\end{equation}
Then the set of functions  $f \in C(Y,[0,1]^{r})$ for which $f^\cF:X \to ([0,1]^{r})^N$ is injective is a dense $G_\delta$ subset in $C(Y,[0,1]^{r})$.
\end{theorem}

In addition to proving Theorem \ref{thm:infinite_groups_Takens}, our main auxiliary theorem has an important application related to the celebrated Takens embedding theorem. We now provide the necessary background.

Consider an experimentalist observing a physical system modeled by a $\Z$-system $(X,T)$.
It often happens that what is observed is the values of $k$ \emph{measurements} $h(x), h(Tx), \ldots, h(T^{k-1} x)$, for a real-valued \emph{observable} $h \colon  X \rightarrow \R$. 
One is led to ask, to what extent the original system can be reconstructed from such sequences of measurements (possibly at different initial points) and what is the minimal number $k$ of \emph{delay-coordinates}, required for a reliable reconstruction. This question has been treated in the literature by what today is known as \emph{Takens-type delay embedding theorems}, essentially stating that the reconstruction of $(X,T)$ is possible for certain observables $h$, as long as the measurements $h(x), h(Tx), \ldots, h(T^{k-1} x)$ are known for \emph{all} $x \in X$ and large enough $k$. Indeed  the first result obtained in this area is the Takens delay embedding theorem (\cite{T81}) --  for a compact
manifold $X$ of dimension $d$, it is a generic property (w.r.t.\ Whitney $C^{2}$-topology) for pairs $(h,T)$, where $T:X\rightarrow X$
is a $C^{2}$-diffeomorphism and $h:X\rightarrow\mathbb{{R}}$ a $C^{2}$-function,
that the $(2d+1)$-delay observation map
\begin{align*}
  [h]_{0}^{2d} \colon X & \longrightarrow \mathbb{{R}}^{2d+1} \\
  x & \longmapsto\big(h(x),h(Tx),\ldots,h(T^{2d}x)\big)
\end{align*}
is an embedding. The importance of Takens' result, as evidenced by the great interest it met among mathematical physicists (see e.g. \cite{HBS15, SYC91, PCFS80}),  lies in the fact that the delay-observation framework which it suggested was and still is widely used by experimentalists (see e.g.~\cite{hgls05distinguishing, KY90,sgm90distinguishing,sm90nonlinear}). There have also appeared various mathematical generalization of the theorem \cite{SYC91, N91, 1999delay, CaballeroEmbed, Rob05, Rob11, Gut16, GQS18, StarkEmbedSurvey, StarkStochEmbed, NV18, kato2023takens}. Furthermore recently a probabilistic point of view has been introduced to the theory \cite{BGS20,baranski2022shroer,baranski2024prediction}.

Note Takens considered a setting where $T:X\rightarrow X$ and $h:X\rightarrow X$ are perturbed in order to achieve embedding. The paper \cite{SYC91} introduced a setup where the dynamics is fixed and only the observable is perturbed. Thus in order to achieve embedding, some conditions on periodic points are necessary\footnote{As an example of this phenomenon consider a t.d.s.\ $(G,X)$  where $\dim (G,X)_1>1$  then for all continuous $h:X\rightarrow \R$ it holds that $h_{|(G,X)_1}$ is not injective and therefore $(h^G)_{|(G,X)_1}$ on is not injective, in particular $h^G$ is not an embedding.}. Whereas \cite{SYC91} assumed some regularity conditions both on the dynamics and the observable\footnote{In \cite{SYC91} it is shown
that given an open set $U\subset\mathbb{{R}}^{k}$,  $C^{1}$-diffeomorphism $T:U\rightarrow U$, a compact subset $A\subset U$ with \textit{lower box dimension} $d$, under certain assumptions on
points of low period, \textit{generically} in $h\in C^{1}(U,\mathbb{{R}})$ it holds that $([h]_{0}^{2d})_{|A}$ is a
topological embedding.}, the paper \cite{Gut16} was the first to study the problem of delayed embedding for fixed dynamics in the purely topological setting. The main theorem of \cite{Gut16} implies  that
for a $\Z$-system $(\Z,X)$
with $\dim(X)=d$ and $\dim(\Z,X)_n<\frac{1}{2}n$ for all $n\leq2d$,  
for a generic function $h\in C(X,[0,1])$,  $[h]_{0}^{2d}$ is an embedding.
In \cite{GQS18} a generalization to $\Z^k$-systems was stated (without proof): For a subgroup $A \subset \mathbb{Z}^k$  define $X_A\subset X$ as
   the space of $x\in X$ satisfying $T^a x=x$ for all $a\in A$. Assume a $\Z^k$-system $(\Z^k,X)$ satisfies
$\dim(X)=d$ and 
$\frac{\dim(X_A)}{[\mathbb{Z}^{k}:A]}<\frac{m}{2}$
for every subgroup $A$ of $\mathbb{Z}^{k}$ with
$[\mathbb{Z}^{k}: A]\leq 2d$,
then for a generic function
 $f\in C(X,[0,1]^{m})$ it holds that
\begin{equation*}
f_{2d}:X\to ([0,1]^{m})^{[0,2d]^{k}\cap\mathbb{Z}^{k}},\;\;x\mapsto (f(ix))_{i\in[0,2d]^{k}\cap\mathbb{Z}^{k}}
\end{equation*}
is an embedding. Our last result is a generalization of the two above mentioned results to the context of finitely generated group actions.

\begin{theorem}\label{thm:takens_fg_group}
    Let $G$ be a finitely generated group, let  $S \subseteq G$ be a finite generating set for $G$, and let $r \ge 1$ be a natural number. Let $(G,X)$ be a topological dynamical system   
     with $\dim(X) < + \infty$
    such that for every $N \in \N$ it holds
\[ \dim(G,X)_N < \frac{rN}{2}.\]
Let $S^{\bullet 0} := \{e_G\}$ and given a natural number $n$, let 
\begin{equation}\label{eq:S_bullet}
S^{\bullet n} = \left\{ s_1 \cdot \ldots s_n ~:~ s_1,\ldots,s_n \in S\right\}    
\mbox{ and }
S^{ \le \bullet n} = \bigcup_{k=0}^n S^{\bullet n}
\end{equation}
\noindent
Let $M$ be the smallest natural number which satisfies $M > \frac{2 \dim(X)}{r}$, and let $F= S^{ \le \bullet (M-1)}$.
Then the set of 
continuous functions $f:X\to [0,1]^{r}$ so that
\begin{equation*}
f^{F}:X\to ([0,1]^{r})^{F},\;\;x\mapsto (f(gx))_{g\in F}
\end{equation*}
is an embedding
is comeagre in $C(X,[0,1]^{r})$.
\end{theorem}

\textbf{Structure of the paper:} Section \ref{sec:prelim} contains basic definitions. 
In Subsection \ref{subsec:Proof of the main theorem}, Theorem  \ref{thm:infinite_groups_Takens} is proven assuming Theorem \ref{thm:gen_Takens}. In Subsection \ref{subsec:sharpness}, Theorem  \ref{thm:sharpness} is proven. Section \ref{sec:Takens Embedding Theorem for finitely generated groups} contains the proof of 
Theorem \ref{thm:takens_fg_group}  assuming Theorem \ref{thm:gen_Takens} as well as Example \ref{ex:finitely generated} showing that if the group is not finitely generated then the conclusion of Theorem  \ref{thm:takens_fg_group} does not necessarily hold. Section \ref{sec:auxiliary theorem}, where Theorem \ref{thm:gen_Takens} is proven, is the main technical part of the article. 

\textbf{Acknowledgments:}
Y.G. was partially
supported by the National Science Centre (Poland) grant
2020/39/B/ST1/02329. M.L. was partially  supported by the Israel Science Foundation grant No. 2196/20. T.M. was partially  supported by the Israel Science Foundation grants No. 1052/18 and 985/23.
 We thank the anonymous referees for helpful comments. 

\section{Preliminaries}\label{sec:prelim}

\subsection{Standing notation} For $k \in \N$ denote
$[k] = \{1,\ldots k\}$ . Let $X$ be a set.
Denote 
\[ X^\Delta:= X \times X \setminus \Delta_X, \mbox{ where: }\Delta_X =\{ (x,x) :~ x \in X\} \subseteq X \times X.\]
Let $\pi_1,\pi_2:X\times X \to X$ denote the canonical projections given by $\pi_j(x_1,x_2)=x_j$ for $j \in [2]$.

\subsection{Dimension}\label{subsec:dimension} Let $X$ be a metric space. Let $\alpha$ and $\beta$ be finite open covers of $X$. 
 We say that $\beta$ \textbf{refines} $\alpha$, denoted $\beta \succ \alpha$, if every member of $\beta$ is contained in a member of $\alpha$.
 The \textbf{join} of $\alpha$ and $\beta$ is defined as $\alpha\vee \beta=\{A\cap B|\ A\in \alpha,B\in \beta\}$. Similarly, one may define the join
$\bigvee_{i=1}^{n} \alpha_{i}$ of any finite collection of open covers $\alpha_{i}$, $i=1,\dots, n,$ of $X$.  
Assume $\alpha$ consists of the open sets $U_{1}, U_{2}\ldots, U_{n}$. Define its \textbf{order}
by $\ord(\alpha)= \max_{x\in X} \sum_{U\in \alpha} 1_{U}(x)-1$ and
let $\Dn(\alpha)$ stand for the minimum order with respect to all
covers $\beta$ refining $\alpha$,
i.e.,\ $\Dn(\alpha)= \min_{\beta\succ\alpha} \ord(\beta)$. 
The \textbf{Lebesgue covering dimension} is defined as 
$$\dim(X)= \sup_{\alpha} \Dn(\alpha),$$
where the supremum is over all finite open covers of $X$. In this article dimension always refers to Lebesgue covering dimension.  Note that Lebesgue covering dimension can only take  values in $\N\cup \{0,\infty\}$.

\subsection{Dynamical systems}
A \textbf{topological dynamical system} (\textbf{t.d.s.}) is a pair  $(G,X)$ where $G$ is a group equipped with the discrete topology\footnote{Note that if the action map $G\times X\rightarrow X$ is continuous w.r.t.\ \textit{some} topology on $G$ then it is continuous w.r.t.\ the discrete topology on $G$. Therefore the assumption that $G$ is equipped with the discrete topology is no restriction.}, $(X,d)$ is a compact metric space\footnote{Sometimes we omit $d$ from the notation.} and $G$ acts on $X$ such that the action map $G\times X\rightarrow X$ given by $(g,x)\mapsto gx$ is continuous.  
A t.d.s.\ $(G,X)$ is also referred to as a $G$-\textbf{system} or a $G$-\textbf{(group) action}.

The \textbf{orbit} of $x$ under $G$ is denoted by $G\cdot x:= \{gx : g \in G \}$. For $F\subset G$, similarly denote $F\cdot x:= \{gx : g \in F \}$. The set of \textbf{periodic points} of \textbf{period}  not bigger than $N$ of $(G,X)$  is denoted by
$$(G,X)_N := \left\{ x \in X:~ |G \cdot x| \le N\right\}.$$ 
Define for $x\in X$
$$\Fix_G(x)=\{g\in G|\,gx=x\}.$$
The  \textbf{kernel} of the action $(G,X)$ is the set
\[\Ker(G,X) := \{ g \in G:~ ~\forall x \in X~~ gx =x \}=\bigcap_{x\in X} \Fix_G(x).\]
Note $\Ker(G,X)$ is a normal subgroup of $G$. A t.d.s.\ $(G,X)$ is called \textbf{faithful} if $\Ker(G,X)$ is the trivial subgroup.

A morphism between
two dynamical systems $(G,X)$ and $(G,Y)$ is given by a continuous
mapping $\varphi:X\rightarrow Y$ which is $G$-\textbf{equivariant}
($\varphi(gx)=g\varphi(x)$ for all $x\in X$ and $g\in G$).
If $\varphi$ is an injective morphism, it is called an \textbf{embedding}.

\subsection{Cubical shits and orbit maps}\label{subsec:cube_shift}
For any space $Z$, the group $G$ acts on the space $Z^G$ by $g(y)_h = y_{hg}$ for all $y \in Z^G$ and $g,h \in G$.  For any space $Z$, the group $G$ acts on the space $Z^G$ by $g(y)_h = y_{hg}$ for all $y \in Z^G$ and $g,h \in G$. When $Z=[0,1]^r$, then this action is referred to as a \textbf{$G$-shift}. When $Z=[0,1]^r$, then the system $(([0,1]^{d})^G,\shift)$ is called the ($G-$) \textbf{$r$-cubical shift}. A continuous mapping $f:X\rightarrow[0,1]^{r}$
induces a continuous $G$-equivariant mapping $f^G:(G,X)\rightarrow(G,([0,1]^{r})^G)$
given by $x\mapsto\big(f(g x)\big)_{g\in G}$, known as the \textbf{orbit-map}.

\subsection{Genericity}\label{subsec:Genericity}
We denote the space of continuous functions from $X$ to $[0,1]^r$ by $C(X,[0,1]^r)$. equipped with the topology of uniform convergence. 
By the Baire category theorem (\cite[Theorem 8.4]{K95})
the space $C(X,[0,1]^{r})$, is a \textbf{Baire
space}, i.e., a topological space where any comeagre set is dense. We refer to a property that holds on a comeagre set of $C(X,[0,1]^{r})$ as \textbf{generic}.

\subsection{Partitions}
Let $\cP$ be a partition of set $S$. For every $s \in S$ denote by $\cP(s)$  the unique element $P \in \cP$ such that $s \in P$.
Let $\cP$ and $\tilde \cP$ be two partitions of the same set. One says $\tilde \cP$ is  \textbf{finer}  than $\cP$, $\tilde \cP \succeq \cP $, if for every $\tilde P \in \tilde \cP$, there exists $P \in \cP$ such that $\tilde P \subseteq P$. 

\subsection{Partition compatible subsets}

\begin{definition}\label{def:partition compatible1}
  Let $X,Y$  be compact metrizable spaces and let $\cF= \big(g_1,\ldots,g_N\big)$   be an $N$-tuple  of continuous functions from $X$ to $Y$.
    Suppose $x\in X$. Define an equivalence relation on $\cF$ by 
    $$g_i\sim_x g_j\, \stackrel{\triangle}{\Leftrightarrow} g_i(x)=g_j(x).$$
    We denote by $x_\cF$ the \textbf{$x$-induced partition of $\cF$} the partition of $\cF$ generated by the equivalence classes of $\sim_x$.
    
 Let $\cP $ be a partition of $\cF$. For $W\subset X$ define the \textbf{$(\cP,\cF)$-compatible subset} by 
\begin{equation}\label{eq:Y_P notation}
W_{\cP} = \left\{ x\in W ~:~  x_\cF=\cP \right\}.
\end{equation}

\end{definition}

Definition \ref{def:partition compatible1} has a natural generalization to functions of two variables which we now present.

\begin{definition}
    Let $X,Y$  be compact metrizable spaces and let $g:X \to Y$ be a function. For $j \in [2]$, define the function $g^{(j)}:X \times X \to Y$ by  
    $$g^{(j)}(x_1,x_2) = g(x_j).$$
\end{definition}

\begin{definition}
 Let $X,Y$  be compact metrizable spaces and let $\cF= \big(g_1,\ldots,g_N\big)$  be a finite ordered set of continuous functions from $X$ to $Y$. Define the \textbf{induced} $2N$-tuple of continuous functions from $X\times X$ to $Y$ by
$$\hat \cF= \big(g_1^{(1)},g_1^{(2)},g_2^{(1)},g_2^{(2)},\ldots,g_N^{(1)},g_N^{(2)}\big)$$
 
\end{definition}

Let $\hat \cP$ be a partition of $\hat \cF$ and $\hat W\subset  X^\Delta$. Following \eqref{eq:Y_P notation}, we may define 
 the \textbf{$(\hat \cP,\hat \cF)$-compatible subset} by 
\begin{equation}\label{eq:Y_P notation2}
\hat W_{\hat \cP} = \left\{ (x,y)\in \hat W ~:~  {(x,y)}_{\hat \cF}=\hat \cP \right\}.
\end{equation}

\subsection{Generating sets}\label{subsec:Generating sets}

Given a subset $S$ of a group $G$ and $n\in \mathbb{N}$ we denote by $S^{-1}$ the set of inverses of elements of $S$ and 
\[
S^{\bullet n} = \{ s_1\cdot \ldots \cdot s_n :~ s_1,\ldots,s_n \in S\}.\]
We denote the identity element of $G$ by $e_G$, and  $S^{\bullet 0} := \{e_G\}$.
\begin{definition}

A subset  $S \subseteq G$ is  a generating set for the group $G$ if 
\[ G = \bigcup_{n=1}^\infty \left(S \cup S^{-1} \cup \{e_G\} \right)^{\bullet n}. \]

\end{definition}


\section{Embedding finite-dimensional systems into cubical shifts}\label{sec:Embedding finite-dimensional systems}
\subsection{Proof of the main theorem}\label{subsec:Proof of the main theorem}

\begin{proof}[Proof of \Cref{thm:infinite_groups_Takens}, assuming \Cref{thm:gen_Takens}]
 Note one may assume w.l.o.g.\ that $(G,X)$ is faithful by considering the induced t.d.s.\ $(G/\Ker(G,X),X)$. If $G$ is a finite group, then one may directly apply \Cref{thm:gen_Takens} with $X=Y$ and $\cF= G$ to deduce the conclusion.

Assume $G$ is infinite. Note one may assume w.l.o.g.\
 that $G$ is  countable. Indeed as  $(G,X)$ is faithful, there is an injective group homomorphism $i:G\rightarrow \homeo(X)$, where $\homeo(X)$ is the group of homeomorphisms of $X$, which is Polish when equipped with the supremum (uniform) metric (\cite[Subsection 9.B (8)]{K95}). Thus one may find a dense (w.r.t.\ the supremum metric) subgroup $G' < i(G)$, where $G'$ is countable. In particular  for every $x \in X$, $|G' \cdot x|= |G \cdot x|$,  and so for every $N \in \N$, 
\[(G,X)_N = \{ x \in X: |G' \cdot x| \le N\}.\]
Hence, by possibly replacing $G$ by $G'$, one may assume that $G$ is at most countable.

Choose $N \in \N$ such that $2\dim(X) < rN$.
Let  $\epsilon >0$ and  $F \subset G$ a finite set. A set $S\subset X$ is called \textbf{$\epsilon$-separated} if for every $s_1\neq s_2\in S$, it holds $d(s_1,s_2)\geq \epsilon$. For $K \subset X$ denote by $\sep_\epsilon(K)$  the maximal cardinality of an $\epsilon$-separated set in $K$.
Define
\[
X^{(F,\epsilon)} = \left\{
x \in X:~ G \cdot x = F \cdot x \mbox { or } \sep_\epsilon(F \cdot x) \ge rN
\right\}.
\]

\noindent
Using the definition of $\sep_\epsilon(\cdot)$, it is not difficult to see that $X^{(F,\epsilon)}$ is a closed subset of $X$. 
Recall notation \eqref{eq:Y_P notation}.
We claim that for every partition $\mathcal{P}$ of $F$ it holds that 
\[
\dim(X^{(F,\epsilon)}_{\mathcal{P}}) <\frac{r}{2}|\mathcal{P}|.
\]
Indeed, if $|\mathcal{P} | < r N$ then for every $x \in X^{(F,\epsilon)}_{\mathcal{P}}$ it holds that $G \cdot x = F \cdot x$, as the condition $\sep_\epsilon(F \cdot x) \ge rN$ implies that $|F \cdot x) |\ge rN$ so the cardinality of the partition $x_\cF$ is strictly bigger than $|\cP|$, in particular $x_\cF\neq \cP$.  Thus in this case $X^{(F,\epsilon)}_{\mathcal{P}} \subseteq X_{|\mathcal{P}|}$, and so
\[
\dim(X^{(F,\epsilon)}_{\mathcal{P}})  \le \dim(X_{|\mathcal{P}|}) < \frac{r}{2}|\mathcal{P}|.
\]
Otherwise, $|\mathcal{P}| \ge r N $ and so
\[
\dim(X^{(F,\epsilon)}_{\mathcal{P}})  \le \dim(X) < \frac{r}{2}|\mathcal{P}|.
\]
For any finite set $F \subset G$ and $\epsilon >0$, by applying  \cref{thm:gen_Takens} with $\mathcal{F}=F$, $Y=X$ and $X= X^{(F,\epsilon)}$, we conclude that the function $f^F:X^{(F,\epsilon)} \to ([0,1]^r)^{F}$ is injective for $f\in \C_{(F,\epsilon)}$
a comeagre subset of $C(X,[0,1]^r)$.

Let $(F_n)_{n=1}^\infty$ be an increasing sequence of finite subsets of $G$ such that $\bigcup_{n=1}^\infty F_n = G$. Then $X= \bigcup_{n=1}^\infty \bigcup_{m=1}^\infty X^{(F_n,\frac{1}{m})}$. Conclude that if $f\in \bigcap_{n=1}^\infty \bigcap_{m=1}^\infty \C_{(F_n,\frac{1}{m})}$, which is a comeagre set then it induces a $G$-equivariant topological embedding $f^{G}:X \to ([0,1]^r)^G$.
\end{proof}

\subsection{Sharpness of the main theorem}\label{subsec:sharpness}

\begin{proof}[Proof of Theorem \ref{thm:sharpness}]
Flores (see \cite{F35}, a more accessible source is \cite[Eng78, §1.11H]{engelking1995theory}) proved that $C_n$, the
union of all faces of dimension less than or equal to $n$ of the $(2n+2)$-simplex (the  convex hull of $2n + 3$ points in $\mathbb{R}^{2n+2}$ being affinely independent) does not embed into $R^{2n}$. Note $\dim(C_n) = n$.    Let $Z$ be a compact metric space  with $\dim(Z)= \lceil \frac{rN}{2}\rceil$ such that $Z$  does not embed into $[0,1]^{rN}$.  Let 
    $$X=(Z\times G/G')\stackrel{\circ}{\cup} \{0,1\}^G,$$
    with $G$ acting trivially on $Z$, by multiplication on $G/G'$ and by shift on $\{0,1\}^G$. Clearly $(G,X)$ is faithful. Denote by $eG'$ the coset of $G/G'$ containing the identity of $G$. Denote by $\vec{0}$ the element of $\{0,1\}^G$ consisting only of zeroes.   Note that for any  continuous function $f:X \to [0,1]^r$, the induced map $f^{G}:X \to ([0,1]^r)^G$ restricted to $Z\times \{eG'\}\times \{\vec{0}\} $ is determined by a continuous map $F:Z\rightarrow ([0,1]^r)^{G/G'}\equiv [0,1]^{rN}$. As $Z$  does not embed into $[0,1]^{rN}$, the map $F$ is not injective and as a consequence $f^{G}$ is not injective. Moreover 
$$Z\times G/G'\times \{\vec{0}\}\subset (G,X)_N\subset X.$$
      Thus $ \lceil \frac{rN}{2}\rceil=\dim Z\leq \dim(G,X)_N \leq \dim X= \dim Z= \lceil \frac{rN}{2}\rceil$ which implies $\dim(G,X)_N = \lceil \frac{rN}{2}\rceil$.
\end{proof}

\section{Takens Embedding Theorem for finitely generated groups}\label{sec:Takens Embedding Theorem for finitely generated groups}

\begin{proof}[Proof of \Cref{thm:takens_fg_group}  assuming Theorem \ref{thm:gen_Takens}]
Let $r \ge 1$, $(G,X)$ and $M$ be as in the statement of the theorem.
  Note that by Theorem \ref{thm:gen_Takens} (used for the case $X=Y$) it is enough to show that with $F= S^{ \le \bullet (M-1)}$, for every partition $\cP$ of $F$ it holds 
  \begin{equation}\label{eq:inequality}
  \dim X_{\cP} < \frac{r}{2}|\cP|.    
  \end{equation}
\noindent  
Let $d = \dim X$, and let $\cP$ be any partition of $F$.
Note that if  $|\cP|\ge M$, then 
 $$|\cP|>\frac{2d}{r},$$
 so inequality \eqref{eq:inequality} holds automatically.  
Otherwise, $|\cP| \le M-1$ and in particular $M\geq 2$. Let $x\in X_{\cP}$. From the definition, it follows $|F\cdot x|=|\cP|\leq M -1$. By \Cref{lem:orbit_seg_orbit} below applied to $M-1$ it follows that 
\begin{equation}\label{eq:FxGx}
G\cdot x=F \cdot x.
\end{equation}
\noindent
We thus have
$$
   X_{\cP}\subset (G,X)_{|\cP|}, 
$$
as $(G,X)_{|\cP|}=\{ x \in X:~ |G x| \le |\cP|\}$ and if $x\in X_{\cP}$ then $|G x|=|F x|=|\cP|$. Thus 
$$
\dim X_{\cP} \leq \dim (G,X)_{|\cP|}<\frac{r}{2}|\cP|, 
$$
as desired.

\end{proof}

\begin{lemma}\label{lem:orbit_seg_orbit}
    Let $G$ be a group that acts on a set $X$, and let $S \subseteq G$ be a finite generating set, and  $M \in \mathbb{N}$. 
    Given a natural number $n$, recall the definition of $S^{\le \bullet n}$ by \eqref{eq:S_bullet}, and let $S^{\le \bullet 0}= \{e_G\}$.
    If 
    $x \in X$ and $|S^{\le \bullet M} \cdot x| \le M$ then $S^{\le \bullet (M-1)} \cdot x = G \cdot x$.
 \end{lemma}

\begin{proof}
Since $S^{\le \bullet (n-1)} \subseteq S^{\le \bullet n}$ for every $n$, it follows that 
\[1=|S^{\le \bullet 0} \cdot x| \le |S^{\le \bullet 1} \cdot x| \le \cdots \le |S^{\le \bullet M} \cdot x| \le M.\]
It follows by the pigeonhole principle that there exists $1\le n \le M$ so that
\begin{equation}\label{eq:Fx}
|S^{\le \bullet (n-1)} \cdot x| = |S^{\le \bullet n} \cdot x|.
\end{equation}
As $S^{\le \bullet n}=\bigcup_{s\in S}sS^{\le \bullet (n-1)}$, it follows that for every $s \in S$ it holds
\begin{equation}\label{eq:sF}
s (S^{\le \bullet (n-1)} \cdot x) \subseteq S^{\le \bullet (n-1)} \cdot x.    
\end{equation}
Note that the map $s\cdot:G\rightarrow G$ given by $g \mapsto s\cdot g$ is an injective map for every $s \in S$. By Equation \eqref{eq:sF}, when restricted to $S^{\le \bullet (n-1)} \cdot x$ it is a self-map  of a finite set. Thus by Equation \eqref{eq:Fx},  it follows that $s (S^{\le \bullet (n-1)} \cdot x) = S^{\le \bullet (n-1)} \cdot x$ for all $s \in S$. Hence we also have $s^{-1} (S^{\le \bullet (n-1)} \cdot x) = S^{\le \bullet (n-1)} \cdot x$ for all $s \in S$. From this it follows that $S^{\le \bullet (n-1)} \cdot x$ is a $G$-invariant subset of $X$ that contains $x$, hence it must contain the $G$-orbit $G \cdot x$. Since $S^{\le \bullet (n-1)} \cdot x \subseteq G \cdot x$, it follows that we have equality. Since $S^{\le \bullet (n-1)} \subseteq S^{\le \bullet M-1} \subseteq G$, it holds
\[
G\cdot x= S^{\le \bullet (n-1)}\cdot x \subseteq S^{\le \bullet (M-1)} \cdot x \subseteq G \cdot x,
\]
So $S^{\le \bullet (M-1)} \cdot x= G \cdot x$ as desired. 
\end{proof}

\begin{remark}\label{rem:sharpness}
From Theorem \ref{thm:takens_fg_group} it follows that for any natural numbers $d,r \in \mathbb{N}$ and finitely generated group $G$  there exists an integer $N_{d,r}= N_{d,r}(G)\in \mathbb{N}$ and a finite set $F_{d,r}=F_{d,r}(G) \subseteq G$ of size at most $N_{d,r}$ so that for any topological dynamical system $(G,X)$ with $\dim X \le d$ the set of 
continuous functions $f:X\to [0,1]^{r}$ so that
\begin{equation*}
f^{F_{d,r}}:X\to ([0,1]^{r})^{F_{d,r}},\;\;x\mapsto (f(gx))_{g\in F_{d,r}}
\end{equation*}
is an embedding
is comeagre in $C(X,[0,1]^{r})$. Indeed, \Cref{thm:takens_fg_group} shows that for a group $G$ generated by a finite set $S$, one may take $N_{d,r}(G)=|S^{\le \bullet (M-1)}|$ with 
$M = \lfloor \frac{2d}{r} \rfloor+1$, defined after \eqref{eq:S_bullet}.
When $r=1$, we have that $G$ is a cyclic group so up to group isomorphism either $G= \Z$ or $G= \Z/m \Z$ for some natural number $m$. In this case   \Cref{thm:takens_fg_group} recovers the result of 
\cite{Gut16}, where it was proven that $F=\{0,1,2,\ldots,2d\}$ is sufficient. The lower bound $r \cdot N_{d,r}(G) \ge 2d+1$ holds for any $d,r \in \mathbb{N}$ and any group $G$,  as evidenced by considering a $d$-dimensional space $X$ which does not embed in $\R^{2d}$ (see the proof of Theorem \ref{thm:sharpness} in Subsection \ref{subsec:sharpness}), thus we conclude that the optimal (minimal) value of $N_{d,r}(\Z)$ is $\lfloor \frac{2d}{r}\rfloor$+1.
It is  interesting  to find the minimal value possible for $N_{d,r}(G)$ for other finitely generated groups $G$ beyond $\Z$ . For example for $G=\Z^2$ and $S=\{(1,0),(0,1)\}$, the proof gives 
$N_{d,r}(\Z^2)=\frac{\lfloor \frac{2d}{r}+1 \rfloor \lfloor
\frac{2d}{r}+2 \rfloor}{2}$.
 One wonders how far this is from the minimal value possible. 
\end{remark}

We present an example showing that if the group is not finitely generated then the conclusion of Theorem  \Cref{thm:takens_fg_group} does not necessarily hold. 
\begin{example}\label{ex:finitely generated}
 Let $1=r_3>r_4>r_5\ldots $ be a sequence of positive numbers with $\lim_{i\rightarrow\infty} r_i=0$.
 Let $S_i$ be the circle of radius $r_i$ around the origin in the plane. Let $X$ be the compact (one-dimensional) metric space
$$
X=\bigcup_{i=3}^\infty S_i \cup \{(0,0)\}\subset \mathbb{R}^2.
$$
 Let $G=F_\infty$, the free group with a non finite countable number of generators. Denote $G=\langle g_i\rangle_{i=1}^\infty$. We define the action of $G$ on $X$ by specifying the action of each generator: the element $g_i$ rotates $S_i$ by $2\pi/i$  and acts as identity otherwise. 
 \noindent
 Note that $(G,X)_1=(G,X)_2=\{(0,0)\}$. Note that for $N\geq 3$,  $(G,X)_N$
is a finite union of circles $S_i$ and the origin. Thus for every $N \in \N$ it holds
\[ 1= \dim(G,X)_N < \frac{1}{2}N. \] 
However for every finite $F\subset G$,  for every  continuous functions $f:X\to [0,1]$ the map
\begin{equation*}\label{deff2d}
f^{F}:X\to [0,1]^{F},\;\;x\mapsto (f(gx))_{g\in F}
\end{equation*}
is not an embedding as one may find a circle $S_i$ on which it equals $f_{|S_i}$.
\end{example}

\section{The main auxiliary theorem}\label{sec:auxiliary theorem}
\subsection{Overview of the proof}

Given a set $\hat Z \subseteq X^\Delta$, denote
\begin{equation*}
\cG_\cF(\hat Z) := \{ f \in C(Y,[0,1]^r)~:~ f^{\cF}(x_1) \ne  f^{\cF}(x_2) ~ \forall (x_1,x_2) \in \hat Z\}.
\end{equation*}

\begin{lemma}\label{lem:cG_open}
Let $\hat Z \subset X^\Delta$  be a compact set. Then the set $\cG_\cF(\hat Z) \subset C(Y,[0,1]^r)$ is open.
\end{lemma}
\begin{proof}
Fix $f \in \cG_\cF(\hat Z)$. Since $\hat Z$ is compact, it follows that there exists $\epsilon >0$ such that for all $(x_1,x_2) \in \hat Z$, the distance between  $f^{\cF}(x_1)$ and $f^{\cF}(x_2)$ is at least $\epsilon$. It follows that the ball of radius $\epsilon/2$ around $f$ in $C(Y,[0,1]^r)$ is contained in $\cG_\cF(\hat Z)$. Q.E.D.
\end{proof}

\begin{lemma}\label{lem:countable_cover_baire}(cf. \cite[Lemma 3.3]{Gut16} 
Suppose there exists a countable collection of compact sets $\hat Z_1,\ldots, \hat Z_n ,\ldots \subset \hat {X}$, such that $\hat{X} = \bigcup_{n=1}^\infty \hat Z_n$ and so that for every $n \in \mathbb{N}$ the set $\cG_\cF(\hat Z_n)$ is dense.
Then 
the set of continuous functions  $f \in C(Y,[0,1]^{r})$ for which $f^\cF:X \to ([0,1]^{r})^\cF$ is injective is a dense $G_\delta$ subset of $C(Y,[0,1]^r)$.
\end{lemma}
\begin{proof}
Given  $f \in C(Y,[0,1]^r)$ observe that $f^{\cF}:X \to ([0,1]^r)^{\cF}$ is injective if and only if $f^{\cF}(x_1) \ne f^{\cF}(x_2)$ for every $(x_1,x_2) \in X^\Delta$,
Since $X^\Delta= \bigcup_{n=1}^\infty \hat Z_n$ it follows that
\[ \left\{ f \in C(Y,[0,1]^r):~ f^{\cF} \mbox { is injective } \right\} = \bigcap_{n=1}^\infty \cG_\cF(\hat Z_n).\]
By assumption $\cG_\cF(\hat Z_n)$ is dense for every $n$. By \Cref{lem:cG_open}$, \cG_\cF(\hat Z_n)$ is open.  By the Baire category theorem, it follows that $\bigcap_{n=1}^\infty \cG_\cF(\hat Z_n)$ is a dense $G_\delta$ subset of $C(Y,[0,1]^r)$.
\end{proof}

Given \Cref{lem:countable_cover_baire}, to conclude the proof of  \Cref{thm:gen_Takens}, it suffices to find a countable cover of  $X^\Delta$ by compact sets $\hat Z_1,\hat Z_2,\ldots \subset X^\Delta $, so that for every $n \in \mathbb{N}$ the set $\cG_\cF(\hat Z_n)$ is dense. 


\subsection{Coherent sets}
Let $(X,d),(Y,d')$  be compact metric spaces and let $\cF= \big(g_1,\ldots,g_N\big)$   be a finite ordered set of continuous injective functions from $X$ to $Y$. Observe that $X^\Delta = \bigcup_{\hat \cP}X^\Delta_{\hat \cP}$, where the union is a finite union over partitions of $\hat \cF$.

\begin{definition}

A  set $\hat Z \subseteq X^\Delta$ is said to be  $\cF$-\textbf{coherent} if there exists a partition $\hat \cF$  such that $\hat Z \subseteq X^\Delta_{\hat \cP}$ and so that for every $(i_1,j_1),(i_2,j_2) \in [N] \times [2]$ such that $\hat \cP(g_{i_1}^{(j_1)}) \ne \hat \cP(g_{i_2}^{(j_2)})$ it holds 
$g_{i_1}^{(j_1)}(\hat Z) \cap g_{i_2}^{(j_2)}(\hat Z) = \emptyset$. 
\end{definition}
\begin{lemma}\label{lem:hat_X_cP_sigma_compact}
Let $\hat \cP$ be a partition  of  $\hat \cF$ 
Then the set $X^\Delta_{\hat \cP}$ is a countable union of compact $\cF$-coherent sets.

\end{lemma}
\begin{proof}
For $(i_1,j_1), (i_2,j_2) \in [N] \times [2]$, denote
\[
X^{\Delta,=}_{(i_1,j_1), (i_2,j_2)}  = \left\{ (x_1,x_2) \in X^\Delta:~  g_{i_1}(x_{j_1})= g_{i_2}(x_{j_2})\right\},
\]

\[
X^{\Delta,\neq}_{(i_1,j_1), (i_2,j_2)}= \left\{ (x_1,x_2) \in X^\Delta:~ g_{i_1}(x_{j_1}) \ne g_{i_2}(x_{j_2})\right\}.
\]

Given a partition $\hat \cP$ of $[N] \times [2]$ denote:
\[
I^{=} = \left\{
\left((i_1,j_1),(i_2,j_2)\right) \in \left([N] \times [2]\right)^2: 
\cP(i_1,j_1) =  \cP(i_2,j_2)
\right\}.
\]
and
\[
I^{\ne} =\left\{
\left((i_1,j_1),(i_2,j_2)\right) \in \left([N] \times [2]\right)^2: 
\cP(i_1,j_1) \ne \cP(i_2,j_2)
\right\}
\]
Note

\[ X^{\Delta}_{\hat \cP} =  \big( \bigcap_{ (i_1,j_1),(i_2,j_2) \in I^{=}} X^{\Delta,=}_{(i_1,j_1), (i_2,j_2)}\big) \bigcap \big( \bigcap_{ (i_1,j_1),(i_2,j_2) \in I^{\ne}} X^{\Delta,\neq}_{(i_1,j_1), (i_2,j_2)}\big).
\]
Define for $n \in \mathbb{N}$ and  $(i_1,j_1), (i_2,j_2) \in [N] \times [2]$  

\[
X^{\Delta,n}_{(i_1,j_1), (i_2,j_2)}  = \left\{ (x_1,x_2) \in X^\Delta:~ d(x_1,x_2) \ge \frac{1}{n} \mbox{ and } g_{i_1}(x_{j_1})= g_{i_2}(x_{j_2})\right\},
\]
and

\[
X^{\Delta,+n}_{(i_1,j_1), (i_2,j_2)}  = \left\{ (x_1,x_2) \in X^\Delta:~ d(x_1,x_2) \ge \frac{1}{n} \mbox{ and } d(g_{i_1}(x_{j_1}) , g_{i_2}(x_{j_2})) \ge \frac{1}{n} \right\}.
\]
Note that each of the sets $X^{\Delta,n}_{(i_1,j_1), (i_2,j_2)}$  and $X^{\Delta,+n}_{(i_1,j_1), (i_2,j_2)}$ is compact.
Also
\[ X_{(i_1,j_1), (i_2,j_2)}^{\ne} = \bigcup_{n=1}^\infty   X_{(i_1,j_1), (i_2,j_2)}^{+n} \mbox{ and } X_{(i_1,j_1), (i_2,j_2)}^{=} = \bigcup_{n=1}^\infty   X_{(i_1,j_1), (i_2,j_2)}^{n}.\]
Thus $X^\Delta_{\hat \cP}$ is a countable union of compact sets. We write:
$$X^\Delta_{\hat \cP}=\bigcup_{i=1}^\infty \hat K_i.$$

Next, we  find for every $(x,y) \in X^\Delta_{\hat \cP}$ a relatively open set $(x,y)\in U_{(x,y)}\subset X^\Delta_{\hat \cP}$ such that $\overline{U}_{(x,y)}\cap X^\Delta_{\hat \cP}$ is $\cF$-coherent.
As $X\times X$ is a \textit{hereditarily Lindelöf space}\footnote{A \textbf{Lindelöf space} is a topological space in which every open cover has a countable subcover. A \textbf{hereditarily Lindelöf space} is a topological space such that every subspace of it is Lindelöf.}, $X^\Delta_{\hat \cP}$ is a Lindelöf space and therefore one may find a countable open subcover consisting of such sets $X^\Delta_{\hat \cP}=\bigcup_{j=1}^\infty U_{(x_j,y_j)}.$
As a subset of an $\cF$-coherent set is also an $\cF$-coherent set, one concludes that

$$X^\Delta_{\hat \cP}=\bigcup_{i=1}^\infty \bigcup_{j=1}^\infty (\hat K_i\cap \overline{U}_{(x_j,y_j)})$$
\noindent
is a countable union of
compact $\cF$-coherent sets.

Finally let us construct $U_{(x_1,x_2)}$ for a given $(x_1,x_2)\in X^\Delta_{\hat \cP}$. Indeed, let $\epsilon>0$ be smaller than the minimal distance between two distinct elements of $\{g_i(x_j): ~i \in [N],~ j \in [2]\}$. By the continuity of the $G$-action, one may find a relatively open set $U_{(x_1,x_2)}\subset X^\Delta_{\hat \cP}$ with $\mathbf{x}=(x_1,x_2)\in U_{(x_1,x_2)}$ such that  for every $\mathbf{z}=(z_1,z_2) \in U_{(x_1,x_2)}$ 
and every $i \in [N]$ and $j \in [2]$,
$$d'(g_i^{(j)}(\mathbf{x}),g_i^{(j)}(\mathbf{z})) \le \frac{\epsilon}{3}.$$ 
By the triangle inequality, for every $\mathbf{z}=(z_1,z_2),\mathbf{\tilde z}=(\tilde z_1,\tilde z_2)  \in \overline{U}_{(x_1,x_2)}$ and every $(i_1,j_1),(i_2,j_2) \in [N] \times [2]$ such that $\hat \cP(i_1,j_1) \ne \hat \cP(i_2,j_2)$ it holds
\[d'(g_{i_1}^{(j_1)}(\mathbf{z}),g_{i_2}^{(j_2)}(\mathbf{\tilde z})) \ge d'(g_{i_1}^{(j_1)}(\mathbf{x}), g_{i_2}^{(j_2)}(\mathbf{x})) - d'(g_{i_1}^{(j_1)}(\mathbf{x}),g_{i_1}^{(j_1)}(\mathbf{z}))-d'(g_{i_2}^{(j_2)}(\mathbf{x}),g_{i_2}^{(j_2)}(\mathbf{\tilde z})),\]
which implies
\[d'(g_{i_1}^{(j_1)}(\mathbf{z}),g_{i_2}^{(j_2)}(\mathbf{\tilde z})) \ge \frac{1}{3}\epsilon.\]
Thus $\overline{U}_{(x_1,x_2)}\cap X^\Delta_{\hat \cP}$  is $\cF$-coherent
\end{proof}

\subsection{Main auxiliary lemma}

In view of \Cref{lem:countable_cover_baire} and \Cref{lem:hat_X_cP_sigma_compact}, in order to prove  \Cref{thm:gen_Takens} it is enough to prove the following lemma. 

\begin{lemma}\label{lem:cG_dense}
Let $(X,d)$, $(Y,d')$  be compact metrizable spaces and let $\cF= \big(g_1,\ldots,g_N\big )$   be a finite ordered set of continuous injective functions from $X$ to $Y$.
Assume that for every partition $\cP$ of $\cF$ the inequality \eqref{eq:dim_inequality} holds.
 Let $\hat \cP$ be a partition  of  $\hat \cF$ and let $\hat Z \subseteq X^\Delta_{\hat \cP}$ be a compact $\cF$-coherent set. 
Then $\cG_\cF(\hat Z_n)$ is a dense subset of $C(Y,[0,1]^r)$.
\end{lemma}


To  prove \Cref{lem:cG_dense} we distinguish between two different types of partitions $\hat \cP$ as follows:



\begin{definition}\label{def:intersective_partition}
A partition $\hat \cP$ of  $[N] \times [2]$ is said to be  \textbf{intersective} if there exists $i_1,i_2 \in [N]$ such that $\hat \cP (i_1,1)= \hat \cP(i_2,2)$. Otherwise, $\hat \cP$ is said to be \textbf{non-intersective}. 
\end{definition}

\begin{lemma}\label{lem:intersective_partition}
Let $\hat \cP$ be  a   partition of $\hat \cF$ identified with  $[N] \times [2]$  such that  $X^\Delta_{\hat \cP} \ne \emptyset$. Then:
\begin{enumerate}
\item For every $i \in [N]$ it holds that $\hat \cP(i,1) \ne \hat \cP(i,2)$.
\item If $\hat \cP$ is an  intersective partition, then there exists a homeomorphism $T:\pi_1(X^\Delta_{\hat \cP}) \to \pi_2(X^\Delta_{\hat \cP})$ such that  for every $i_1\neq i_2 \in [N] $ satisfying $\hat \cP(i_1,1)=\hat \cP(i_2,2)$ and $x \in \pi_1(X^\Delta_{\hat \cP})$, it holds $g_{i_1}(x)= g_{i_2}(T(x))$.
\end{enumerate}
\end{lemma}
\begin{proof}
\begin{enumerate}
\item Suppose that for some $i \in [N]$ it holds that $\hat \cP(i,1) = \hat \cP(i,2)$. As $X^\Delta_{\hat \cP} \ne \emptyset$, one may find $ (x_1,x_2) \in X^\Delta_{\hat \cP}$. Then $x_1 \ne x_2$ and $g_{i}(x_1) = g_i(x_2)$, contradicting the injectivity of $g_i$.

\item Let $\hat \cP$ be an intersective partition, and let   $i_1\neq i_2 \in [N]$  with   $\hat \cP(i_1,1) = \hat \cP(i_2,2)$. Then in particular $g_{i_1}(\pi_1(X^\Delta_{\hat \cP}))=g_{i_2}(\pi_2(X^\Delta_{\hat \cP}))$. Since $g_{i_1},g_{i_2}$ are injective continuous functions on $X$, they induce hoemomorphisms between  $X$ and $g_{i_1}(X),g_{i_2}(X) \subseteq Y$ respectively. Then $T:= g_{i_2}^{-1}\circ g_{i_1}\mid_{\pi_1(X^\Delta_{\hat \cP})}$ is a homeomorphism between $\pi_1(X^\Delta_{\hat \cP})$ and $\pi_2(X^\Delta_{\hat \cP})$. Now suppose  $i_1',i_2' \in [N]$ also satisfy  $\hat \cP(i_1',1)=\hat \cP(i_2',2)$.  
Choose any $x_1 \in \pi_1(X^\Delta_{\hat \cP})$. 
Our goal is to show that  $g_{i_1'}(x_1)= g_{i_2'}(T(x_1))$.
By definition,  $x_1 \in \pi_1(X^\Delta_{\hat \cP})$ implies that there exists $x_2 \in \pi_2(X^\Delta_{\hat \cP})$ such  $(x_1,x_2) \in X^\Delta_{\hat \cP}$.
The fact that $(x_1,x_2) \in X^\Delta_{\hat \cP}$ implies  that $g_{i_1}(x_1)=g_{i_2}(x_2)$ and also $g_{i_1'}(x_1)=g_{i_2'}(x_2)$.  The condition  $g_{i_1}(x_1)=g_{i_2}(x_2)$ is equivalent to $T(x_1)=x_2$. So indeed $g_{i_1'}(x_1)= g_{i_2'}(T(x_1))$.

\end{enumerate}
\end{proof}

The key tool from dimension theory that is used in the proof of Lemma \ref{lem:cG_dense} is a result  known as ``Ostrand's theorem'', which we now recall:

\begin{theorem}[Ostrand's theorem, \cite{MR177391}]
A compact metric space $X$ satisfies $\dim(X) <  n$ if and only if for every $\epsilon >0$ and $k \ge 0$ there exist $n+k$ families $\cC_1,\ldots,\cC_{n+k}$, such that each $\cC_i$ consists of pairwise disjoint closed subsets of $X$ of diameter at most $\epsilon$, and so that every element of $X$ is covered by at least  $k+1$ elements of $\bigcup_{j=1}^{n+k} \cC_j$.
\end{theorem}

The statement here is equivalent but not identical to the original one in \cite[Theorem 1]{MR177391}, where it has been used to extend previous results of Kolmogorov and Arnold on Hilbert's 13th problem \cite{MR0111809,MR0111808}. See \cite[Theorem 2.4]{MR2418302} for another application of Ostrand's theorem.

\begin{definition}
Let $f:X \to \mathbb{R}$ be a function and $\epsilon >0$. Denote by $ \delta_f: (0,\infty) \rightarrow  [0,\infty]$  the function given by:
\begin{equation*}
\delta_f(\epsilon) := \sup\left\{ \delta' \ge 0 :~ \forall x_1,x_2 \in X,\ d(x_1,x_2) \le \delta' \rightarrow |f(x_1)-f(x_2)| \le \epsilon \right\}.
\end{equation*}

\end{definition}
Clearly, $\delta_f(\epsilon)$ is finite for any bounded $f:X \to \mathbb{R}$.
If $f:X \to \mathbb{R}$ is continuous, then (by compactness of $X$) $\delta_f(\epsilon) >0$ for any $\epsilon >0$.

\begin{lemma}\label{lem:cont_approx_disjoint_open_covers}
Let $Y$ be a compact metric space, and let $f_1,\ldots,f_r \in C(Y,[0,1])$ and $\epsilon >0$ be given. Let $\cC_1,\ldots,\cC_r$ be sets of subsets of $Y$, where each $\cC_\ell$ is a set  of pairwise disjoint closed subsets of $Y$, each having diameter smaller than
 $\delta_{f_\ell}(\epsilon/2)$.
 Then there exists  functions 
$\tilde f_1,\ldots,\tilde f_r \in C(Y,[0,1])$ such that:
\begin{itemize}
\item[(a)] $\| \tilde f_\ell- f_\ell\|_\infty \le \epsilon$ for all $\ell \in [r]$
\item [(b)] For every $\ell \in [r]$, $C,C' \in \cC_\ell$ $x \in C$, and $x'\in C'$, if $C \ne C'$ then $\tilde f_\ell(x)\ne \tilde f_\ell(x')$. 
\item[(c)] For every $\ell_1,\ell_2 \in [r]$ with $\ell_1 \ne \ell_2$ and every $x_1 \in \bigcup \cC_{\ell_1}$, $x_2 \in \bigcup \cC_{\ell_2}$ it holds $\tilde f_{\ell_1}(x_1) \ne \tilde  f_{\ell_2}(x_2)$.
\end{itemize}
\end{lemma}
\begin{proof}
By the definition of $\delta_{f_\ell}(\epsilon/2)$, for every $\ell \in [d]$ and every $C \in \cC_\ell$ it holds that the diameter of $f_\ell(C)$ is at most $\epsilon/2$. For every $\ell \in [r]$ choose a function $v_\ell:\cC_\ell \to [0,1]$ so that:
\begin{itemize}
    \item[($\alpha$)] The distance between $v_\ell(C)$ and $f_\ell(C)$ is at most $\epsilon/2$ for every $\ell \in [r]$ and every $C \in \cC_\ell$.
    \item[($\beta$)] $v_\ell:\cC_\ell \to [0,1]$ is injective for every $\ell \in [r]$.
    \item[($\gamma$)] For every  $\ell_1,\ell_2 \in [r]$ with $\ell_1 \ne \ell_2$ $v_{\ell_1}(\cC_{\ell_1}) \cap v_{\ell_2}(\cC_{\ell_2}) = \emptyset$.
\end{itemize}
Such a function $v_\ell:\cC_\ell \to [0,1]$ exists  because each $\cC_\ell$ is finite.

For every $\ell \in [r]$, let $\hat f_\ell:\bigcup \cC_\ell \to [0,1]$ be defined by $\hat f= \sum_{C \in \cC_\ell}v_\ell(C) 1_{C}$. Since $\cC$ is a family of pairwise disjoint sets, each $\hat f_\ell$ is locally constant, hence continuous. Because the distance between $v_\ell(C)$ and $f_\ell(C)$ is at most $\epsilon/2$ for every $\ell \in [r]$ and every $C \in \cC_\ell$ and the diameter of $f_\ell(C)$ is at most $\epsilon/2$, it follows that $|f_\ell(x) - \hat f_\ell(x)| \le \epsilon$ for all $\ell \in [r]$.
By  the Tietze extension theorem, each $\hat f_\ell$ can be extended to a continuous function $\tilde f_\ell \in C(Y,[0,1])$ so that $\| \tilde f_\ell - f_\ell\| \le \epsilon$, so property \emph{(a)} is satisfied. Property \emph{(b)} for $\tilde f_\ell$ follows directly from \emph{($\beta$)}, and property \emph{(c)} follows directly from \emph{($\gamma$)}.

\end{proof}

We use the following ad-hoc combinatorial  lemma:

\begin{lemma}\label{lem:three_part_graph}
Let $V_1,V_2,W$ be  finite sets such that $|V_1| \ge |V_2|$,  
 and let  $F_1:W \to V_1$, $F_2:W \to V_2$ be surjective functions.
Suppose $V_1^* \subseteq V_1$, $V_2^* \subseteq V_2$ satisfy $|V_1^*| > \frac{|V_1|}{2}$ and $|V_2^*| > \frac{|V_2|}{2}$.
Then at least one of the following holds:
\begin{enumerate}
\item[(A)] There exists $w \in W$ such that $F_1(w) \in V_1^*$ and $F_2(w) \in V_2^*$. 
\item[(B)] There exists $w,w' \in W$ such that $F_2(w)=F_2(w')$ , $F_1(w) \ne F_1(w')$ and $F_1(w),F_1(w') \in V_1^*$ .
\end{enumerate}
\end{lemma}
\begin{proof}
As  $F_1:W \to V_1$ is surjective, there exists an injective function $\psi:V_1 \to W$ such that $F_1(\psi(v_1))=v_1$ for all $v_1 \in V_1$. Let $W^* = \psi(V_1^*)$.
By injectivity of $\psi$, $|W^*| = |V_1^*|$. Using $|V_1^*| > \frac{|V_1|}{2}$ and $|V_1| \ge  |V_2|$ we conclude that  $|W^*| > |V_2| /2$. 
Now $|V_2 \setminus V_2^*| < \frac{|V_2|}{2}$, so $|V_2 \setminus V_2^*| <|W^*|$.
If 
$F_2(W^*) \cap V_2^* \ne \emptyset$ we are in case (A).  Otherwise, $F_2(W^*) \subseteq V_2 \setminus V_2^*$, and by the inequality  $|V_2 \setminus V_2^*| <|W^*|$ it follows that  the restriction of $F_2$ to $W*$ is not injective. This  implies that there exist
$w,w' \in W^*$  such that $w \ne w'$ and  $F_2(w)=F_2(w')$. Since $F_1$ is injective on $W^*$ by construction, and $F_1(W^*) = V_1^*$, it follows that  $F_1(w) \ne F_1(w')$ and $F_1(w),F_1(w') \in V_1^*$ , so we are in case (B).
\end{proof}

\begin{lemma}\label{lem:comb_no_intersection}
Let $V_1,V_2,W$ be  finite sets, let $Y$ be an arbitrary set,  and let  $F_1:W \to V_1$, $F_2:W \to V_2$, $\phi_1:V_1 \to Y$, $\phi_2:V_2 \to Y$ be functions, and $V_1^* \subseteq V_1$, $V_2^* \subseteq V_2$. Further, suppose that the restrictions of $\phi_1$ and $\phi_2$ to $V_1^*$ and $V_2^*$ respectively, are both injective and that $\phi_1(F_1(w)) \ne \phi_2(F_2(w))$ for every $w \in F_1^{-1}(V_1^*) \cap F_2^{-1}(V_2^*)$, and that at least one of the statements (A) and (B) from \Cref{lem:three_part_graph} hold. Then $\phi_1 \circ F_1 \ne \phi_2 \circ F_2$.
\end{lemma}
\begin{proof}
Statement (A) from \Cref{lem:three_part_graph} implies that that there exists $w \in F_1^{-1}(V_1^*) \cap F_2^{-1}(V_2^*)$, so by assumption for such $w$ it holds $\phi_1(F_1(w)) \ne \phi_2(F_2(w))$, and so in this case  $\phi_1 \circ F_1 \ne \phi_2 \circ F_2$.

Now suppose statement (B) from \Cref{lem:three_part_graph} holds. Namely, we assume that there exists $w,w' \in W$ such that $F_2(w)=F_2(w')$, $F_1(w) \ne F_1(w')$ and $F_1(w),F_1(w') \in V_1^*$. By assumption, $\phi_1$ is injective on $V_1^*$. So $\phi_1(F_1(w)) \ne \phi_1(F_1(w'))$. But  $F_2(w)=F_2(w')$ 
implies $\phi_2(F_2(w))=\phi_2(F_2(w'))$, so in this case we again conclude that $\phi_1 \circ F_1 \ne \phi_2 \circ F_2$.
\end{proof}

We are now ready to prove \Cref{lem:cG_dense}.
\begin{proof}[Proof of \Cref{lem:cG_dense}]
Let $\hat \cP$ be a  partition of  $ \hat \cF$ identified with $[N] \times [2]$, and let $\hat Z \subseteq \hat{X}_{\hat \cP}$ be a compact $\cF$-coherent set.
 Let $f =\left( f_1,\ldots,f_r\right)\in C(Y,[0,1]^r)$ and $\epsilon >0$ be arbitrary. Our goal is to find $\tilde f \in \cG_\cF(\hat Z)$ such that $\| \tilde f -f\|_\infty < \epsilon$.
 Let 
\[ \delta: = \min_{ \ell \in [r]} \frac{1}{2}\delta_{f_\ell}(\epsilon /2).\]
 By the compactness of $X$ and the continuity of the maps $g_i:X \to Y$  one may find  $\eta >0$ such that for all $i \in [N]$ and $x,x' \in X$  satisfying $d(x,x') < \eta$ it holds $d'(g_i(x),g_i(x'))< \delta$.
For $j \in [2]$, let $Z_j:= \pi_j (\hat Z)$ denote the projection of $\hat Z \subseteq X \times X$ into the $j$'th copy of $X$. 
For $j \in [2]$ let $\cP_j$ denote the partition of $[N]$ defined by $\cP_j(i_1)=\cP_j(i_2)$ iff $\hat \cP(i_1,j)=\hat \cP(i_2,j)$, and let
  $M_j=|\cP_j|$ for $j \in [2]$. Assume without loss of generality that $|M_1| \ge |M_2|$.
Then it holds $Z_j \subseteq X_{\mathcal{P}_j}$ and so by  the inequality  \eqref{eq:dim_inequality}, $\dim(Z_j) < \frac{r}{2}M_j$ for $j \in  [2]$.

For $j \in [2]$, write $\cP_j = \{P_1^{(j)},\ldots,P_{M_j}^{(j)}\}$.
Then for every $j \in [2]$ and $t \in [M_j]$  there exists a function $\tilde g_{t,j} \in C(Z_j,Y)$ such that
$g_i\mid_{Z_j} = \tilde g_{t,j}$ for every $i \in P_{t}^{(j)}$.
By Ostrand's theorem and the condition $\dim(Z_2) < \frac{r}{2}M_2$, one may find families of sets 
\[\cC^{(2)}_{t,\ell} \mbox{ for }   t \in [M_2], \ell \in [r],\]
such that each $\cC^{(2)}_{t,\ell}$ is a family of pairwise disjoint closed subsets of $Z_2$ having diameter smaller than $\eta$, and so that every $x \in Z_2$ is covered by at least  $\left(M_2 - \dim(Z_2) +1\right) > \frac{r}{2}M_2$ elements of $\bigcup_{ t \in [M_2], \ell \in [r]}\cC^{(2)}_{t,\ell}$.
For every  $t \in [M_2]$ and $\ell \in [r]$ define:

\[\tilde \cC^{(2)}_{t,\ell} := \left\{ \tilde g_{t,2}(C):~ C \in \cC^{(2)}_{t,\ell}\right\}.\]

As $\hat Z \subseteq X^\Delta_{\hat \cP}$ is $\cF$-coherent, it holds that $\tilde g_{t,2}(Z_2) \cap  \tilde g_{t',2}(Z_2) = \emptyset$ for every $t \ne t'$ $t,t' \in [M_2]$. It follows that for each $\ell \in [r]$ it holds that $\bigcup_{t \in [M_2]} \tilde \cC^{(2)}_{t,\ell}$ is a collection of pairwise disjoint closed subsets of $Y$. By the choice of $\eta$, the diameter of each of these sets is less than $\delta$.

The next step of the proof splits into  two cases:
\begin{itemize}
\item \emph{Case 1: The partition $\hat \cP$ is non-intersective}. In this case, by Ostrand's theorem and the condition $\dim(Z_1) < \frac{r}{2}M_1$, one may find families of sets 
\[\cC^{(1)}_{t,\ell} \mbox{ for }   t \in [M_1], \ell \in [r],\]
such that each $\cC^{(1)}_{t,\ell}$ is a family of pairwise disjoint closed subsets of $Z_1$ having diameter smaller than $\eta$, and so that every $x \in Z_1$ is covered by  at least  $\left(M_1 - \dim(Z_1) +1\right) > \frac{r}{2}M_1$ elements of $\bigcup_{ t \in [M_1], \ell \in [r]}\cC^{(1)}_{t,\ell}$.
For every  $t \in [M_1]$ and $\ell \in [r]$ define:

\[\tilde \cC^{(1)}_{t,\ell} := \left\{ \tilde g_{t,1}(C):~ C \in \cC^{(1)}_{t,\ell}\right\}.\]

As $\hat Z \subseteq X^\Delta_{\hat \cP}$ is $\cF$-coherent and $\hat \cP$  is non-intersective it holds that $\tilde g_{t,1}(Z_1) \cap  \tilde g_{t',1}(Z_1) = \emptyset$ for every $t \ne t'$ $t,t' \in [M_1]$ and also $\tilde g_{t_1,1}(Z_1) \cap  \tilde g_{t_2,2}(Z_2) = \emptyset$ for every $t_1 \in [M_1]$ and $t_2 \in [M_2]$.
For every $\ell \in [r]$ let 

\[\tilde \cC_\ell :=  \bigcup_{t  \in [M_1]}\tilde \cC^{(1)}_{t,\ell} \cup \bigcup_{t \in [M_2]} \tilde \cC^{(2)}_{t,\ell} .\]

Then by the discussion above $\tilde \cC_\ell$ is  
a collection of pairwise disjoint compact subsets of $Y$ having diameter less than $\delta$.

\item  \emph{Case 2: The partition $\hat \cP$ is intersective}.
Let $I \subseteq [M_2]$ denote the set of indices $t_2\in [M_2]$ which corresponds to the ``intersecting'' partition elements of $\cP_2$. Namely, 
whenever $i \in [N]$ satisfies $\cP_2(i) = P_{t_2}^{(2)}$ then  there exists $i' \in [N]$ such that $\hat \cP(i,2)=\hat \cP(i',1)$.
 Thus there exists an injective function $\zeta : I \to [M_1]$ such that $\tilde g_{t_2,2}(x_2)=\tilde g_{\zeta(t_2),1} (x_1)$ for every $(x_1,x_2) \in \hat Z$. 
By \Cref{lem:intersective_partition} in this case there exists a homeomorphism $T:Z_1 \to Z_2$ such that $\tilde g_{t_2,2} \circ T = \tilde g_{\zeta(t_2),1}$ for every $t_2 \in I$.  In particular, in this case, $\dim Z_1 = \dim Z_2$. By our assumption $|M_1| \ge |M_2|$. So one may extend $\zeta$ in an arbitrary fashion to an injective function $\zeta:[M_2] \to [M_1]$. For every $t \in \zeta([M_2])$ let 
\[ \cC^{(1)}_{t,\ell} =  \left\{ T^{-1}(C):~ C \in \cC^{(2)}_{t,\ell}\right\}.\]
For each $t \in [M_1] \setminus  \zeta([M_2])$ let $\cC^{(1)}_{t,\ell}$ be an arbitrary finite collection of pairwise disjoint closed subsets of $Z_1$ having diameter smaller than $\eta$.

For every  $t \in [M_1]$ and $\ell \in [r]$ define:

\[\tilde \cC^{(1)}_{t,\ell} := \left\{ \tilde g_{t,1}(C):~ C \in \cC^{(1)}_{t,\ell}\right\}.\]

Note that $\tilde \cC^{(1)}_{\zeta(t),\ell} = \tilde \cC^{(2)}_{t,\ell}$ for every $t \in I$.

For every $\ell \in [r]$ let 

\[\tilde \cC_\ell := \bigcup_{t  \in [M_1] \setminus \zeta([M_2])}\tilde\cC^{(1)}_{t,\ell}  \cup   \bigcup_{t \in [M_2]} \tilde\cC^{(2)}_{t,\ell} .\]

As in case 1,  $\tilde \cC_\ell$ is  
a collection of pairwise disjoint compact subsets of $Y$ having diameter less than $\delta$.

\end{itemize}

By  the choice of $\delta$, using  \Cref{lem:cont_approx_disjoint_open_covers}, there exists  functions $\tilde f_1, \ldots,\tilde f_{r}:Y \to [0,1]$  such that 
\begin{itemize}
\item[(a)] $\| \tilde f_\ell- f_\ell\|_\infty \le \epsilon$ for all $\ell \in [r]$
\item [(b)] For every $\ell \in [r]$, $C,C' \in \tilde \cC_\ell$ $x \in C$, and $x'\in C'$, if $C \ne C'$ then $\tilde f_\ell(x)\ne \tilde f_\ell(x')$. 
\item[(c)] For every $\ell_1,\ell_2 \in [r]$ with $\ell_1 \ne \ell_2$ and every $x_1 \in \bigcup \cC_{\ell_1}$,$x_2 \in \bigcup \cC_{\ell_2}$ it holds $\tilde f_{\ell_1}(x_1) \ne \tilde f_{\ell_2}(x_2)$.
\end{itemize}

To complete the proof, we  show that  $\tilde f \in \cG_\cF(\hat Z)$. 
To prove that  $\tilde f \in \cG_\cF(\hat Z)$, we need to show that for every $(x_1,x_2) \in \hat Z$ there exists $i \in [N]$ and $\ell \in [r]$ such that 
\[ \tilde f_\ell (g_i(x_1)) \ne  \tilde f_\ell (g_i(x_2)).\]

Choose any $(x_1,x_2) \in \hat Z$. 

For $j \in [2]$, denote
$V_j = [M_j] \times [r] \times \{j\}$, 
and 
\[V_j^* = \{ (t,\ell,j) \in [M_j] \times [r] \times \{j\}:~ x_j \in \cC^{(j)}_{t,\ell}\}.\]
Denote $W:= [N] \times [r]$.
Define functions  $F_j:W \to V_j$ by 
\[ F_j(i,\ell) = (t,\ell) ~\Leftrightarrow~ (i,j) \in P_t^{(j)},~ j \in [2],~t\in [M_j], i \in [N],~ \ell \in [r].\]
Direct inspection reveals that the assumptions of \Cref{lem:three_part_graph} are satisfied with  
$V_1,V_2$,$V_1^*,V_2^*$,  $F_1:W \to V_1$ and $F_2:W \to V_2$ as above.
For each $j \in[2]$ define  $\phi_j:V_j \to Y$  by 
\[
\phi_j(t,\ell) = \tilde f_\ell(\tilde g_{t,j}(x_j)) \mbox{ for } \ell \in [r],~ j \in [2],~ t \in [M_j].
\]

Then the restrictions of $\phi_1$ and $\phi_2$ to $V_1^*$ and $V_2^*$ respectively, are both injective by the properties \emph{(b)} and \emph{(c)} of the functions  $\tilde f_\ell$.
 \Cref{lem:intersective_partition} implies that $\hat \cP(i,1) \ne \hat \cP(i,2)$ for any $i \in [N]$. So 
 $\phi_1(F_1(w)) \ne \phi_2(F_2(w))$ for every $w \in F_1^{-1}(V_1^*) \cap F_2^{-1}(V_2^*)$,  by injectivity of each of the functions $v_\ell$.

By \Cref{lem:comb_no_intersection} it holds that $\phi_1 \circ F_1 \ne \phi_2 \circ F_2$. This precisely implies that $\tilde f^{\cF} (x_1) \ne \tilde f^{\cF}(x_2)$, completing the proof.

\end{proof}

\bibliographystyle{alpha}
\bibliography{library,universal_bib, universal_bib2}
\end{document}